\documentclass[11pt,reqno]{amsart}
\usepackage{amsmath, latexsym, amsfonts, amssymb,amsthm, amscd,epsfig,enumerate, color}
\advance\hoffset -.75cm

\usepackage{amsfonts}
\usepackage{latexsym}
\usepackage{amsmath}
\usepackage{amssymb}

\newcommand{\Z}{\mathbb Z}
\newcommand{\N}{\mathbb N}
\newcommand{\Prob}{\mathbb P}
\newcommand{\pr}{\mathbb P}

\newcommand{\ident}{{\mathchoice {\rm 1\mskip-4mu l} {\rm 1\mskip-4mu l}
{\rm 1\mskip-4.5mu l} {\rm 1\mskip-5mu l}}}

\newtheorem{teo}{Theorem}[section]
\newtheorem{lem}[teo]{Lemma}
\newtheorem{cor}[teo]{Corollary}
\newtheorem{rem}[teo]{Remark}
\newtheorem{pro}[teo]{Proposition}

\newtheorem{exmp}[teo]{Example}

\begin{document}

\title[Frog models]
{Local and global survival for nonhomogeneous random walk systems on $\Z$}

\author[D.~Bertacchi]{Daniela Bertacchi}
\address{D.~Bertacchi,  Universit\`a di Milano--Bicocca
Dipartimento di Matematica e Applicazioni,
Via Cozzi 53, 20125 Milano, Italy
}
\email{daniela.bertacchi\@@unimib.it}

\author[F.~P.~Machado]{F\'abio Prates Machado}
\address{F.~P.~Machado, Instituto de Matem\'atica e Estat\'\i stica,
Universitade de S\~ao Paulo, Brasil.}
\email{fmachado\@@ime.usp.br}

\author[F.~Zucca]{Fabio Zucca}
\address{F.~Zucca, Dipartimento di Matematica,
Politecnico di Milano,
Piazza Leonardo da Vinci 32, 20133 Milano, Italy.}
\email{fabio.zucca\@@polimi.it}

\date{}

\begin{abstract}
 We study an interacting random walk system on $\Z$ where at time 0 there is an
active particle at 0 and one inactive particle on each site $n\ge1$. Particles
become active when hit by another active particle. Once activated,
the particle starting at $n$ performs an asymmetric, translation invariant, nearest neighbor random walk 
with left jump probability $l_n$. We give conditions for global survival,
local survival and infinite activation both in the case where all particles are
immortal and 
 in the case
where particles have geometrically distributed lifespan (with parameter depending on the starting
location of the particle). More precisely, once activated, the particle at $n$ survives at each step
with probability $p_n \in [0,1]$.
In particular, in the immortal case, we prove a 0-1 law for the probability of local
survival when all particles drift to the right. Besides that, we give sufficient conditions
for local survival or local extinction when all particles drift to the left.
In the mortal case, we provide sufficient conditions for global survival, local
survival and local extinction (which apply to the immortal case with mixed
drifts as well). Analysis of explicit examples is provided: we describe completely the phase diagram
in the cases $1/2- l_n \sim \pm 1/n^\alpha$, $p_n=1$ and $1/2-l_n \sim \pm 1/n^\alpha$, $1-p_n \sim 1/n^\beta$ 
(where $\alpha, \beta>0$).
\end{abstract}

\maketitle



\baselineskip .6 cm

\section{Introduction}
\label{sec:intro}

We study an interacting random walk system on $\Z$ where at time 0
there is one active particle at 0 and one inactive particle at each vertex
of $\N\setminus \{0\}=\{1, 2, \ldots\}$ (our results apply also if at time $0$
there are empty vertices in $\N\setminus \{0\}$, see Section~\ref{sec:mortal}).
Particles become active if an active
particle jumps to their location.
 The behavior of the system depends on two
sequences $\{l_n\}_{n\ge0}$ and $\{p_n\}_{n\ge0}$ of numbers in $(0,1)$ and $[0,1]$ respectively.
The particle which at time 0 was at $n$,
once activated, has a geometrically distributed lifespan with parameter $1-p_n$
and while alive performs a nearest neighbor random walk with probability $l_n$
of jumping to the left and $1-l_n$ of jumping to the right. If $p_n=1$ we say that
the particle is immortal, otherwise it is mortal.
We are interested in establishing, depending on the parameters, whether the
process \textit{survives globally, locally} and if there is \textit{infinite
activation} or not. Local and global survival have been studied for
 several processes; among these it is worth mentioning the
\textit{Contact Process} and the \textit{Branching Random Walks}
in continuous and discrete time (see for instance
\cite{cf:BZ, cf:BZ2, cf:BZsurvey, cf:MachadoMenshikovPopov, cf:Pem, cf:PemStac1,
cf:Z1}).

To be precise, if $L_0$ is the event that
site 0 is visited infinitely many times,
 we say that there is \textit{local survival} if $L_0$ has positive probability and
\textit{almost sure local survival} if $L_0$ has probability 1. When there is no local
survival, that is, when $L_0$ has probability zero, we also say that there is
\textit{local extinction}.
We say that there is \textit{global survival}
if, with positive probability, at any time there is at least one active particle, and we say
that there is \textit{infinite activation} if, with positive probability,
at arbitrarily large times there are particles which turn from inactive to active.

This process can be seen as a model for information or disease spreading:
every active particle has some information
and it shares that information
with all particles it encounters on its way.
In the last decade, different versions of this model have been studied, often under the name
\textit{frog model} or \textit{egg model}.
In \cite{cf:TW99} the authors prove almost sure local survival for a system of simple random
walkers on $\Z^d$. This result has been extended in \cite{cf:P01} to the case of a random
initial configuration ($d \ge 3$) and in \cite{cf:GS} for random walks on $\Z$
with right drift.
Shape theorems on $\Z^d$ 
can be found in \cite{cf:AMP02b,cf:AMPR01}.
Phase transitions for the model where particles have a $\mathcal G(1-p)$-distributed lifespan,
are investigated in \cite{cf:AMP02b,cf:FMS04,cf:LMP05,cf:P03}.
Recently, in \cite{cf:LMZ10},
global survival of an asymmetric inhomogeneous random walk system on $\Z$
 has been studied (in that model
particles die after $L$ steps without activation).

Here is a sketch of the formal construction of our process. Let $(\Omega, \mathcal{F}, \pr)$
a probability space and $\{\{Z_n^i\}_{n \in \N}\}_{i \in \N}$ a family of independent random walks on
$\N \cup \{D\}$ (where $D \not \in \N$ is an absorbing state that we call \textit{death state})
such that $\{Z_n^i\}_{n \in \N}$ starts from $i$ (that is, $Z^i_0=i$). At each step, if the $i$-th walker is not at $D$,
 it jumps to the left with probability $p_i l_i$,
to the right with probability $p_i(1-l_i)$ and to $D$ with probability $1-p_i$. Once in $D$ the walkers stay there indefinitely
with probability $1$. If $p_i=1$ then the $i$-th walker is immortal, while if $p_i=0$ then the $i$-th walker goes to $D$ immediately
(it is like having no particle at $i$ since it does not participate to the evolution).
Our frog model is a collection of dependent walks $\{\{X_n^i\}_{n \in \N}\}_{i \in \N}$ constructed iteratively as follows.
Let $X_n^0=Z^0_n=0$ for all $n \in N$. Suppose we defined $\{\{X_n^i\}_{n \in \N}\}_{i\le N}$; let us define
$\{X_n^{N+1}\}_n$.
Let $T_{N+1}=\min\{k \colon \exists i \in \{0, \ldots, N\},\, X_k^i=N+1\}$ (where $\min (\emptyset):=+\infty$).
Hence, for all $\omega \in \Omega$,
\[
X_n^{N+1}(\omega):=
\begin{cases}
N+1 & n < T_{N+1}(\omega) \le \infty,\\
Z_{n-T_{N+1}(\omega)}^{N+1}(\omega) &  T_{N+1}(\omega) \le n<\infty.\\
\end{cases}
\]

Here is the outline of the paper and of its main results.
We first deal, in Section~\ref{sec:immortal}, with
the case where all particles are immortal (that is, $p_n=1$ for all $n\ge0$).
It is obvious that in this case there is always
global survival, but infinite activation
is trivial only in the case where at least one particle has $l_n\le 1/2$.
Local survival is nontrivial unless $l_n=1/2$ for some $n \in \N$. In order to understand what
the difficulties one encounters are, think of the case where all particles
drift to the right (we refer to this situation as the \textit{right drift}
case): infinite activation is guaranteed but local survival
is not. 
Theorem~\ref{th:localsurv}(1) states that, in this
case, the probability of local survival obeys a 0--1 law. Roughly speaking
(see Corollary~\ref{cor:condition1}) in the \textit{right drift} case, if $l_n\uparrow\frac12$
sufficiently fast, then we have almost sure local survival, otherwise we have local
extinction. Corollary~\ref{cor:condition1} gives conditions which quantify how fast
the convergence to $\frac12$ should be in order to ensure local survival.
On the other hand, if all particles drift to the left (\textit{left drift}
case), local survival and infinite activation
have the same probability (Theorem~\ref{th:localsurv}(2)).
Proposition~\ref{pro:rumor} and Remark~\ref{rem:Binfty} provide sufficient conditions
for infinite activation (thus also for local survival) in the \textit{left drift} case.
The idea in Remark~\ref{rem:Binfty} is that with positive probability
there is a simple ``chain reaction''
where the initial particle visits a certain site to its right, then the particle there
visits a certain site to its right, and so on.
Example~\ref{exmp:block} shows that
the fact that every possible chain reaction has probability $0$
is not necessary to ensure almost sure finite activation  (thus Proposition~\ref{pro:rumor} is indeed
a stronger result).
Theorem~\ref{pro:blocks} states that if certain subsequences of $\{l_n-1/2\}_{n\ge0}$
(when $l_n<1/2$) and of $\{n(1/2-l_n)\}_{n\ge0}$ (when $l_n>1/2$) stay in some $\ell^p$ space,
then there is local survival.
By Proposition~\ref{pro:RWapproach}, if $\inf_{n \in \N} l_n>1/2$ then there is no infinite
activation (thus no local survival). Examples~\ref{exmp:firstpossibility}
and \ref{exmp:secondpossibility} show that if $\inf_{n \in \N} l_n=1/2$
nothing can \textit{a priori} be said about infinite activation. Examples~\ref{exmp:rightd}
and~\ref{exmp:firstpossibility} together describe completely the phase diagram of the
immortal particle model where $1/2-l_n \sim \pm 1/n^{\alpha}$, $\alpha>0$.

Section~\ref{sec:mortal} is devoted to the case where each particle may be mortal
and has geometrical lifespan with parameter $1-p_n$, $p_n\in[0,1]$.
If $p_n <1$ for all $n$ then any particle disappears almost surely after a finite number of step,
thus global
survival is no longer guaranteed and, even if all particles have right drift, so is
infinite activation. Indeed in this case global survival and infinite activation
have the same probability.
In Subsection~\ref{subsec:geometricglobal} we give 
sufficient conditions for global extinction (Proposition~\ref{excor:geomglobalextinction})
and for global survival (Theorem~\ref{excor:geomglobalsurvival}).
In particular we show that
to survive globally it is necessary that $\limsup_n p_n=1$ and it is sufficient that certain
subsequences of  $\{1-p_n\}_{n\ge0}$ and of $\{l_n-1/2\}_{n\ge0}$ stay in some $\ell^p$ space.
In Subsection~\ref{subsec:geometriclocal}
we deal with the problem of local survival of the process and give some sufficient conditions
on the speed of decay of $\{1-p_n\}_{n\ge0}$ and of $\{l_n-1/2\}_{n\ge0}$
which imply local extinction (Theorem~\ref{pro:geomlocalextinction}) or local survival
(Theorem~\ref{excor:geomlocalsurvival}).
Corollary~\ref{cor:power} shows how our results apply to the case
$1-p_n\sim 1/n^\beta$ ($\beta>0$) and  $1/2-l_n \sim \pm 1/n^\alpha$ ($\alpha>0$)
completely describing the phase diagrams in these cases.

All the proofs are to be found in Section~\ref{sec:Proofs}, while in Section~\ref{sec:final}
we comment on some further questions which could be investigated: one possible generalization
is the study of the process in random environment (see Theorem~\ref{th:localsurvRE}).

\section{Immortal particles}
\label{sec:immortal}

In this section, all particles are immortal, that is, $p_n=1$ for all $n\ge0$.
This assumption guarantees global survival, nevertheless local survival
and infinite activation need additional conditions on the sequence $\{l_n\}_{n\ge0}$.
Clearly, if for some $n\in\N$, $l_n=1/2$ 
then there is local survival and infinite activation
(with positive probability the initial particle reaches $n$ and the
random walk associated to $n$ is recurrent). Therefore in this section we assume that $l_n\neq1/2$ for all $n$.

Let $A_n$ be the event that
the particle at $n$ is activated and ever visits $0$ and $B_n$
the event that the particle at $n$ is activated sooner or later. Clearly $A_n\subseteq B_n$ 
and $\Prob(A_n) >0$ for all $n \in \N$ (since $l_n \in (0,1)$ for all $n \in \N$).
Note that $\{A_n\text{ i.o.}\}\subseteq L_0$ and
$\Prob(L_0 \setminus \{A_n\text{ i.o.}\})=0$. Moreover if there exists $n$ such that $l_n<1/2$ then 
$\Prob(B_m|B_n)=1$ for all $m>n$, thus in this case there is infinite activation.

For any choice of $\{l_n\}_{n\ge0}$, by standard random walk computations,
\[
\Prob(A_n|B_n)
=\begin{cases}
1 & \text{if } l_n>1/2;\\
\left(\frac{l_n}{1-l_n}\right)^n &\text{if }l_n<1/2.
 \end{cases}
\]
Let $B_\infty=\bigcap_{n=1}^\infty B_n$ be the event that
all the particles are
activated sooner or later; $B_\infty$ represents infinite
activation.
Our first goal is to find conditions on the sequence
$\{l_n\}_{n\ge0}$ which
guarantee local survival. The knowledge of the behaviour
of the system with
a fixed sequence helps characterizing many other
sequences. Namely, by
coupling it is not difficult to show that in the
\textit{right drift} case
if we have local survival with  $\{l_n\}_{n\ge0}$, then
there is local
survival with any  $\{l^\prime_n\}_{n\ge0}$ such that
$l^\prime_n\ge l_n$
for all $n$.
Conversely, in the \textit{left drift} case, if we have
local survival with
$\{l_n\}_{n\ge0}$, then there is local
survival with any  $\{l^\prime_n\}_{n\ge0}$ such that
$l^\prime_n\le l_n$
for all $n$.

The following theorem includes the particular case of
\cite[Theorem 2.2]{cf:GS} when the initial condition is one particle per site a.s.~(there $l_n=1-p$ for all $n$).
Theorem~\ref{th:localsurv} characterizes the \textit{right drift} case 
in terms of the sequence $\{l_n\}_{n\ge0}$ and shows that 
in the \textit{left drift} case
the probability of local survival is equal to the probability of
infinite activation. 

\begin{teo}\label{th:localsurv}
\begin{enumerate}
\item Suppose that $l_n<1/2$ for all $n$ 
(\textit{right drift} case).
The probability of local survival obeys a 0-1 law:
\[
\Prob(A_n\text{ i.o.})=
    \begin{cases}
    0 &\text{if }\sum_{n \in \N} \left(\frac{l_n}{1-l_n}\right)^n<+\infty;\\
    1 &\text{otherwise.}
        \end{cases}
\]
\item
Suppose that $l_n>1/2$ for all $n$ (\textit{left drift} case).
Then $\Prob(B_\infty \bigtriangleup (A_n\text{ i.o.}))=0$
(where $\bigtriangleup$ denotes symmetric difference of sets).
\end{enumerate}
\end{teo}

The following corollary gives some
conditions which are easy to check and that imply convergence or divergence of
the characterizing series of Theorem~\ref{th:localsurv}(1). We denote by $\log(\cdot)$ the 
natural logarithm.

\begin{cor}\label{cor:condition1}
In the \textit{right drift case} ($l_n<1/2$ for all $n$):
\begin{enumerate}
 \item if $\liminf_n n(1/2-l_n)<+\infty$ then $\Prob(A_n \text{ i.o.})=1$;
\item if \[\frac{n(1/2-l_n)}{\log(n)} \le  \frac{1}{4- \log(n)/n} \] for every sufficiently large $n$,  then $\Prob(A_n \ i.o.)=1$.
\item if \[\frac{n(1/2-l_n)}{\log(n)} \ge  \frac{1+\beta \log(\log(n))/\log (n)}{4- 2(\log(n)+ \beta \log(\log(n)))/n} \] 
for some $\beta>1$ and every sufficiently large $n$,  then $\Prob(A_n \ i.o.)=0$. 
\item if there exists $\lambda<4$ such that
 $\sum_n \exp(-\lambda n(1/2-l_n)) < +\infty$
 then  $\Prob(A_n \text{ i.o.})=0$.
\end{enumerate}
\end{cor}


The previous result implies that, even if $l_n\uparrow\frac12$, there is no local survival
if $l_n$ converges to 1/2 too slowly.
Moreover, when $\lim_n n (1/2-l_n)/\log(n)$ exists then there is a critical threshold, namely $1/4$,
separating local survival from local extinction. More precisely,
if $\lim_n n (1/2-l_n)/\log(n)<1/4$ then there is local survival (Corollary~\ref{cor:condition1}(2)), 
while if $\lim_n n (1/2-l_n)/\log(n)>1/4$ then
there is extinction (Corollary~\ref{cor:condition1}(3)). 
Finally, if $\lim_n n (1/2-l_n)/\log(n)=1/4^-$ (that is, the sequence converges from below) 
then there is local survival again  (Corollary~\ref{cor:condition1}(2)). As for the behavior
when $\lim_n n (1/2-l_n)/\log(n)=1/4^+$ we can have either local survival or local extinction:
indeed, if the equality in Corollary~\ref{cor:condition1}(2) holds then we have local survival, while
if the equality in Corollary~\ref{cor:condition1}(3) holds we have local extinction.

It is easy to extend Theorem~\ref{th:localsurv} to
the cases where there are both particles
with right drift and particles with left drift, as we note in the following remark,
which allows us to focus only on the two ``pure'' cases where all
particles drift towards the same direction.
\begin{rem}\label{rem:mixed}
If all but a finite number of particles have right drift
then by Theorem~\ref{th:localsurv}(1)
$\sum_{n\colon l_n<1/2} \left(\frac{l_n}{1-l_n}\right)^n<+\infty$ implies local
extinction.
If the series diverges, then $\Prob(A_n\text{ i.o.})=\Prob(B_\infty)=\Prob(B_j)$ 
 where $j=\min\{n\colon l_n<1/2\}$, hence we have local survival, since $\Prob(B_j)>0$
  (to prove that $\Prob(A_n\text{ i.o.}|B_j)=1$ when the series diverges, one has to mimic 
the proof of Theorem~\ref{th:localsurv}(1)).
\\
On the other hand, if there is an infinite number of particles
with left drift and at least one with right drift,
 then again we have local survival, since
$\Prob(A_n\text{ i.o.})=\Prob(B_j)$, where $j=\min\{n\colon l_n<1/2\}$.
\end{rem}
The previous remark and Corollary~\ref{cor:condition1} are useful for the analysis of the following example.

\begin{exmp}\label{exmp:rightd}
 Let $1/2-l_n \sim 1/n^\alpha$ as $n \to \infty$ (where $\alpha >0$).
Since $l_n < 1/2$ for every sufficiently large $n \in \N$, by Remark~\ref{rem:mixed}
it is enough to consider the case $l_n<1/2$ for all $n \in \N$.
It is clear that if $\alpha \ge 1$ then $n(1/2-l_n)/\log(n)  \to 0$ thus, 
by Corollary~\ref{cor:condition1}(2),
there is 
local survival; conversely, if $\alpha \in (0,1)$ then
$n (1/2-l_n)/\log(n) \to +\infty$
hence there
is local extinction according to Corollary~\ref{cor:condition1}(3).
\end{exmp}

%

In the \textit{left drift} case, Theorem~\ref{th:localsurv}(2) tells us that local
survival and infinite activation have the same probability. Thus it is interesting
to find conditions for $\Prob(B_\infty)>0$. The first proposition is the following;
its proof makes use of a coupling between the frog model and the rumor process (see~\cite{cf:JMZ}
for the formal definition).

\begin{pro}\label{pro:rumor}
 Suppose that
that $l_n > 1/2$ for all $n \ge0$. If, for some increasing sequence $\{n_k\}_{k \in \mathbb{N}}$,
\begin{equation}\label{eq:Binfty1}
  \sum_{k=0}^\infty  \prod_{i=0}^{n_k} \left (
1 - \left ( \frac{1-l_i}{l_i}\right )^{n_{k+1}-i}
\right ) < +\infty
\end{equation}
then
$\Prob(B_\infty)>0$.
\end{pro}

The main idea in the proof of the previous proposition is to provide a positive lower bound for the
probability that, for each $k$, at least one particle between $0$ and $n_k$ visits site $n_{k+1}$.
These events, indexed by $k$, are not independent since they involve overlapping sets of particles.  

If we consider disjoint sets of particles, we still have dependence of the associated walks 
$\{\{X_n^i\}_{n \in \N}\}_{i \in \N}$, since, to start moving, each particle needs to be activated by another walker.
Nevertheless, disjoint sets $I$ and $J$ have corresponding walks $\{\{Z_n^i\}_{n \in \N}\}_{i \in I}$ and
 $\{\{Z_n^i\}_{n \in \N}\}_{i \in J}$ which are independent by construction. 
This motivates the following strategy, used throughout
the whole paper.
We choose a sequence of non-void, pairwise disjoint sets $\{\mathcal{B}_n\}_{n \in \N}$,
$\mathcal{B}_n \subset \N$ that we call \textit{blocks}. The general idea is to estimate the probability
that particles in different blocks perform specific different tasks.
One way of choosing the sequence $\{\mathcal{B}_n\}_{n \in \N}$ 
is to partition $\N$ into disjoint intervals and to choose, for every $n \in \N$,
 $\mathcal{B}_n$ as a subset of the $n$-th interval (as in the following figure).

\begin{center}
 \includegraphics[width=9truecm]{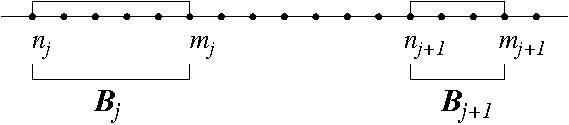}
\\{Figure 1.}
\end{center}

As an example,
consider, for each $j\ge 0$, the event $B^\prime_j:=$``at least
one particle in $\mathcal{B}_j$ visits all the sites of the of $\mathcal{B}_{j+1}$''.
 Using the explicit construction of the process, it is easy to show that $\bigcap_{j \in \N} B^\prime_j$
is the same event if we consider the frog model $\{\{X_n^i\}_{n \in \N}\}_{i \in \N}$ or the process
$\{\{Z_n^i\}_{n \in \N}\}_{i \in \N}$. 
Using the second process, however, allows us to take advantage
of the fact that we are dealing with an intersection of independent events.
Clearly $\bigcap_{j \in \N} B^\prime_j$ is a subset of $B_\infty$, thus if it has positive probability,
then there is infinite activation, which, in the \textit{left drift} case, also means local
survival. 
The probability of $\bigcap_{j \in \N} B^\prime_j$ is strictly positive if and only if
$
  \sum_{j=0}^\infty  \prod_{i \in \mathcal{B}_j} \left (
1 - \left ( \frac{1-l_i}{l_i}\right )^{m_{j+1}-i}
\right ) < +\infty
$
where $m_j:=\max(\mathcal{B}_j)$ 
(the 
proof mimics the one of Proposition~\ref{pro:rumor});
hence the previous inequality implies $\pr(B_\infty)>0$. In this case, the block decomposition
simply gives us 
a corollary of Proposition~\ref{pro:rumor}; in what follows we 
use this method 
to obtain sufficient conditions for local and global survival in the general case.

A special case of block decomposition is given by $\mathcal{B}_j:=\{n_j\}$ for a suitable 
increasing sequence $\{n_j\}_{j \in \N}$ and it is treated in the following remark.

\begin{rem}\label{rem:Binfty}
Suppose that all particles have left drift ($l_n>1/2$ for all $n$).
If there exists an increasing sequence $\{n_j\}_{j \in \N}$ such that 
$\prod_{j\in \N} ((1-l_j)/l_j )^{n_{j+1}-n_j}>0$ or, equivalently, 
\begin{equation}\label{eq:Binfty2}
 \sum_{j=0}^\infty (n_{j+1}-n_{j}) \frac{2 l_{n_{j}}-1}{l_{n_{j}}}<+\infty.
\end{equation}
then $\Prob(B_\infty)>0$. 
To be precise, we are exploiting the fact
that if there exists a subsequence $\{n_j\}_{j\ge0}$ in $\N$, such that
the event ``the $n_j$-th particle visits the $n_{j+1}$-th vertex, for all $j\ge0$''
has positive probability, then also $B_\infty$ has positive probability.
%
%

One may think that we need to test infinitely many sequences looking for the one satisfying equation~\eqref{eq:Binfty2}.
Actually, there is one particular sequence $\{\bar n_j\}_{j \in \N}$ such that 
if either it is finite or 
it is infinite but does not satisfy equation~\eqref{eq:Binfty2},
then there are no sequences satisfying equation~\eqref{eq:Binfty2}.
Define the increasing sequence
$\{\bar n_j\}_{j \in \N}$ as the collection of all points satisfying $h(n)<h(n-1)$, 
where $h(n):=\min\{(2l_k-1)/l_k: k\le n\}$. 
The claim follows from Lemma~\ref{lem:test}(4).
In particular if $\{l_n\}_{n\ge0}$ is 
decreasing
then $\bar n_j=j$ for all $j \in \N$.
%

Finally there might be infinite activation  even if there are no 
sequences  satisfying equation~\eqref{eq:Binfty2}, as Example~\ref{exmp:firstpossibility}
shows.
\end{rem}

For all $k\ge0$, the event ``the $n_k$-th particle visits the $n_{k+1}$-th vertex''
implies the event ``at least one particle between $0$ and $n_k$ visits the $n_{k+1}$-th vertex'', hence
equation~\eqref{eq:Binfty2} implies equation~\eqref{eq:Binfty1}. Nevertheless the condition
given by Remark~\ref{rem:Binfty} is easier to check.

The block argument leads to other nice sufficient conditions for local survival.
The following result is a particular case of Theorem~\ref{excor:geomlocalsurvival}(3) when $p_n=1$ for all $n \in \N$ and it
deals simultaneously with the mixed case.


\begin{teo}\label{pro:blocks}
Suppose that there exists $L \in \N$  and a sequence of pairwise disjoint sets $\{\mathcal{B}_n\}_{n \in \N}$ such that 
$\#\mathcal{B}_n=L$ and $\sup_{n \in \N} (\max(\mathcal{B}_{n+1})-\min(\mathcal{B}_n))<+\infty$. Define
$\mathcal{O}:=\bigcup_{j \in \N} \mathcal{B}_{2j+1}$, $\mathcal{E}:=\bigcup_{j \in \N} \mathcal{B}_{2j}$.
If 
\[
 \sum_{n \in \mathcal{E} \colon  l_n>1/2} (l_n-1/2)^L < +\infty, \quad 
\sum_{n \in \mathcal{O} \colon l_n < 1/2} n^L(1/2-l_n)^L < +\infty
\]
then there is local survival. 
\end{teo}

In the previous theorem, the idea is to make sure that particles in even labelled blocks $\mathcal{E}$ take
care of the activation process while particles in odd labelled blocks $\mathcal{O}$ take
care of visiting the origin (at least one particle for each block). 

In particular, we have that a sufficient condition for local survival in the left-drift (resp.~right-drift) 
case is the existence of $L, m \in \N \setminus\{0\}$ and of an increasing sequence $\{n_j\}_{j \in \N}$, $n_{j+1}-n_j \le m$
such that
\[
 \sum_{j \ge 0} (l_{n_j}-1/2)^L < +\infty, \quad \Big (\textrm{resp.~}\sum_{j \ge 0} {n_j}^L(1/2-l_{n_j})^L < +\infty \Big).
\]
Indeed, if we take $L_1=Lm$, then every interval 
$[nL_1,(n+1)L_1-1]$ 
contains at least $L$ vertices in  $\{n_j\}_{j \in \N}$. For all $n \in \N$, take exactly $L$ of these vertices in the $n$-th interval 
and  
define $\mathcal{B}_n$ as the set containing those vertices. Theorem~\ref{pro:blocks} yields the conclusion.
Other, more powerful, conditions can be derived from Theorem~\ref{excor:geomlocalsurvival}(2) and (3) by taking $p_n=1$ for all $n \in \N$. 

Here are two examples which make use of Theorem~\ref{pro:blocks}. The first one is the left-drift counterpart
of Example~\ref{exmp:rightd}.

\begin{exmp}\label{exmp:firstpossibility}
 Take $l_n - 1/2 \sim 1/n^\alpha$ as $n \to \infty$
where $\alpha > 0$; 
hence $l_n >1/2$ for every sufficiently large $n \in \N$ and
if $L>1/\alpha$ then
$\sum_n \left ( l_n -1/2 \right )^L <\infty$. Thus,
by Theorem~\ref{pro:blocks}, we have local survival and
$\Prob(B_\infty)>0$ for all $\alpha>0$.
\end{exmp}

The second example shows that the sufficient condition given in Remark~\ref{rem:Binfty}
is not necessary.

\begin{exmp}\label{exmp:block}
 Let $l_0=1/2+\varepsilon$, $\varepsilon \in (0,1/2)$ and $l_n=1/2+1/j^2$ if $j^3 \le n<(j+1)^3$. By Theorem~\ref{pro:blocks}
there is local survival, indeed $\sum_{n \in \N} (l_n-1/2)^L= \varepsilon^L+ \sum_{j \ge 1} (3j^2+3j+1)/j^{2L}$ which converges
for $L$ sufficiently large. On the other hand $\{l_n\}_{n \in \N}$ is nonincreasing, hence by Remark~\ref{rem:Binfty} there exists
a sequence $\{n_j\}_{j \in \N}$ satisfying equation~\eqref{eq:Binfty2} if and only if $\sum_{n \ge 0} (l_n-1/2)/l_n<\infty$: this is
false since 
\[
 \sum_{n \ge 0} \frac{l_n-1/2}{l_n} \ge \sum_{j \ge 1} (3j^2+3j+1) \frac{1/j^2}{1/2+1/j^2} =\infty.
\]
\end{exmp}

So far, in the \textit{left drift} case, we have seen only sufficient conditions for
$\Prob(B_\infty)>0$.
We now give a sufficient condition for
$\Prob(B_\infty)=0$, whose proof makes use of a random walk approach.

\begin{pro}\label{pro:RWapproach}
In the left drift case
 ($l_n>1/2$ for all $n$),
if $\liminf_{n \to  \infty} l_n>1/2$ then $\Prob(B_\infty)=0$ and there is
local extinction.
\end{pro}

In the \textit{left drift} case,
if $\inf_{n \in \N} l_n = 1/2$ then  both $\Prob(B_\infty)>0$ or
$\Prob(B_\infty)=0$ are possible. Indeed, Example~\ref{exmp:firstpossibility} shows that survival is possible,
while extinction is shown by Example~\ref{exmp:secondpossibility} where 
$l_n\downarrow 1/2$ slowly enough that $\Prob(B_\infty)=0$.

\begin{exmp}\label{exmp:secondpossibility}
Consider two decreasing sequences $\{q_i\}_{i \in \N}$, $\{\delta_i\}_{i \in \N}$ such that
$q_i\downarrow 1/2$ and $\delta_i \downarrow 0$. Let $n_0=0$, we construct
$\{n_k\}_{k \in \N}$ iteratively.
Suppose we defined $n_i$ for all $i \le k$. Let $\mathcal{M}_k$ be the random
walk system with left jump probabilities $\{\widehat l_i{(k)}\}_{i \in \N}$ where
\[
\widehat l_i{(k)} =
\begin{cases}
 q_j & \textrm{if } j \in \{n_j, \ldots, n_{j+1}-1\}, \ \forall j <k,\\
 q_k & \textrm{if } i \ge n_k.
\end{cases}
\]
We know that, since $\inf_i \widehat l_i{(k)} >1/2$,
by Proposition~\ref{pro:RWapproach}, almost surely for the model $\mathcal{M}_k$
we will have only a finite number of activations. 
Hence it is possible to find $n_{k+1}$ large enough such
that, with probability at least $\delta_{k+1}$, no particles in $\{i\colon i \ge n_{k+1}\}$ will be
activated in the model $\mathcal{M}_k$.

Now for all $j\in\N$ define $l_j=q_k$ if $j \in \{n_k, \ldots, n_{k+1}-1\}$
and denote by $\mathcal{M}$ the corresponding random walk system. Clearly
$l_n\downarrow 1/2$.
For the model $\mathcal{M}$,
 the 
probability of activating particles
in the $(k+1)$-th block 
is at
most $1-\delta_{k+1}$ (for all $k$). This is due to the fact that,
since $l_i=\widehat l_i{(k)}$ for all $i <n_{k+1}$, before the activation of
the $n_{k+1}$-th particle,
$\mathcal{M}$ and $\mathcal{M}_k$ have the same behavior.
Hence, 
with probability 1, sooner or later there
will be no new activations and $\Prob(B_\infty)=0$.
\end{exmp}

\section{Particles with geometrical lifespan}
\label{sec:mortal}

We now suppose that the particle at $n$ survives, at each step, with probability $p_n$, thus, once activated,
it has a lifespan which is $\mathcal{G}(1-p_n)$-distributed.
More precisely, the probability that the lifetime of the $n$-th particle equals $k \ge 1$ is $p_n^{k-1}(1-p_n)$
if $p_n \in [0,1)$ and $0$ if $p_n=1$
(in this last case the lifetime is infinite a.s.).
The main differences between the immortal particle case are that here global survival
is not guaranteed (but it has the same probability as the event of infinite activations)
and that the knowledge of the drift, \textit{a priori}, plays a minor role. Indeed particles
with right drift will activate a finite number of sites almost surely and particles
with left drift have 
probability strictly smaller than $1$ of visiting
the origin.
We assume in this whole section that $p_n \in [0,1]$ for all $n$, that is,
that every particle can be mortal as well as immortal (clearly $p_0>0$ otherwise the process would not start at all).
Observe that if $p_n=0$ for some $n \in \N$ then those particles do not participate 
in the evolution of the system therefore it is like having a system with empty vertices.
On the other hand, if $p_n=1$ for some $n$, then there is global survival, since
there is a positive probability of activating those particles.
We also assume that $l_n \in (0,1)$ for all $n\in \N$ and $l_n \not = 1/2$ for all $n$ such that $p_n=1$.

%

\begin{rem}
By using a coupling between the mortal process and a process with immortal particles
with the same sequence $\{l_n\}_{n \in \N}$, it is clear that all sufficient conditions for
global or local extinction given in Section~\ref{sec:immortal} extend to the mortal case
(compare Examples~\ref{exmp:rightd} and \ref{exmp:firstpossibility} and Corollary~\ref{cor:power}). 
Indeed the probability of local (resp.~global) survival is nondecreasing with 
respect to $\{p_n\}_{n \in \N}$.
As for the dependence of the probabilities of survival with respect to $\{l_n\}_{n \in \N}$ the 
discussion before Theorem~\ref{th:localsurv} applies.
%
Even in the mortal case the most interesting situations are
$\sup l_n=1/2$ (if \textit{right drift}) and $\inf l_n=1/2$ (if \textit{left drift}).
Indeed if $\sup l_n<1/2$ then, according to Corollary~\ref{cor:condition1}, then 
there is local extinction even for an immortal particle system,
thus there is no local survival in the mortal case. 
If $\inf l_n>1/2$, by Proposition~\ref{pro:RWapproach},
even in the immortal case we activate only a finite number of particles almost surely,
thus there is global extinction in the mortal case.
\end{rem}

\subsection{Conditions for global survival/extinction}\label{subsec:geometricglobal}

In this case global survival is not guaranteed and has the same probability
of $B_\infty$, that is, the event of infinite activation.
In order to activate infinitely many sites
we need the action of infinitely many particles.
It is no longer true, as it was in the case of immortal particles,
that it suffices that there exists a particle with right drift
to have infinite activation (that particle would still be activated
with positive probability but, in the mortal case, it will almost surely
activate only a finite number of particles).
The following results give sufficient conditions for global extinction and global survival respectively.
\begin{pro}\label{excor:geomglobalextinction} 
If $\sup_n p_n<1$ then $\Prob(B_\infty)=0$ and
there is no global survival almost surely.
\end{pro}

\begin{teo}\label{excor:geomglobalsurvival} 
If there exists $L \in \N$ and a sequence of pairwise disjoint  sets $\{\mathcal{B}_n\}_{n \in \N}$ such that 
$\#\mathcal{B}_n=L$, $\sup_{n \in \N} (\max(\mathcal{B}_{n+1})-\min(\mathcal{B}_n))<+\infty$ and
\[
\sum_{k \in \bigcup_{n \in \N} \mathcal{B}_{n}} (1-p_{k})^{L/2}<+\infty, \quad \sum_{k \in \bigcup_{n \in \N} 
\mathcal{B}_{n}\colon  p_k l_k >1/2} (l_{k}-1/2)^{L} 
< +\infty, 
\]
 then there is global survival.
\end{teo}

A sufficient condition 
for global survival is the existence of $L,m \in \N\setminus \{0\}$ and of an increasing sequence 
$\{n_j\}_{j \in \N}$, such that $n_{j+1}-n_j \le m$ and 
$\sum_{j\ in \N} (1-p_{n_j})^{L/2} < +\infty$, $\sum_{j\colon p_{n_j} l_{n_j} >1/2} (l_{n_j}-1/2)^L < +\infty$
(see the discussion after Theorem~\ref{pro:blocks}).

\subsection{Conditions for local survival/extinction}\label{subsec:geometriclocal}

From now on we deal with local survival and local extinction. 
%
The first assertion follows from a Borel-Cantelli argument (as Theorem~\ref{th:localsurv}(1))
once we note that 
\begin{equation}\label{eq:ngoes0}
 \pr(A_n|B_n)= \left(\frac{1-\sqrt{1-4p_n^2l_n(1-l_n)}}{2p_n(1-l_n)}
\right)^n,
\end{equation}
which follows from random walk computations (see Section~\ref{sec:Proofs}).

\begin{teo}\label{pro:geomlocalextinction}
If $\sum_{n} \left(
\frac{1-\sqrt{1-4p_n^2l_n(1-l_n)}}{2p_n(1-l_n)}
\right)^n <+\infty$ then $\Prob(A_n\ i.o.)=0$. 
In particular
\begin{enumerate}
\item let $ \Delta_n=1-p_n$ and $\delta_n=1/2-l_n$; 
if $\delta_n \wedge 0 \to 0$ as $n \to \infty$ and \\
$\liminf_{n \to \infty} n \big (2\delta_n+\sqrt{2\Delta_n+4\delta_n^2} \big )/\log(n)>1$ 
then  $\Prob(A_n\ i.o.)=0$;
 \item
if $\sum_n p_n^n(1-(1-2l_n)^+)^n <+\infty$ 
then
$\Prob(A_n\ i.o.)=0$
(where $(1-2l_n)^+=(1-2l_n)\vee0$); 
\item
 if $\sum_n p_n^n <+\infty$ (for instance, if $ \sup_n p_n<1$)
then $\Prob(A_n\ i.o.)=0$ for any choice of $\{l_n\}_n$.
\end{enumerate}
\end{teo}

The tricky part is finding conditions for local survival:
on the one hand we need that all particles get
activated sooner or later and on the other hand that infinitely many
of them visit the origin.
To avoid dealing with situations where a particle is required
both to visit a certain number of sites and the origin,
we exploit once again the idea of choosing pairwise disjoint blocks in $\N$.
Some blocks will take care of activation and
the others of local survival.
We choose a sequence of pairwise disjoint sets $\{\mathcal{B}_n\}_{n \in \N}$ of cardinality $L$. The event where,
for all 
$n \in \N$ such that $\max(\mathcal{B}_{2n+2}) >
\max(\mathcal{B}_{2n})$, at least one of the particles in $\mathcal{B}_{2n}$ visits 
the rightmost vertex of $\mathcal{B}_{2n+2}$  
and an infinite number of particles in $\bigcup_{n \in \N} \mathcal{B}_{2n+1}$ visits the origin
is a subset of the event of local survival. Thus a sufficient
condition for local survival is that this event has positive probability.
In the following figure we picture the case $\mathcal{B}_n:=[nL, (n+1)L-1]$.
\begin{center}
  \includegraphics[width=11truecm]{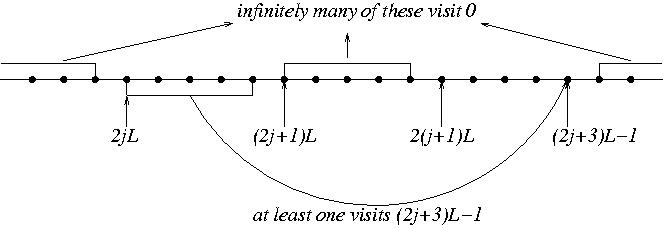}
\\{Figure 2.}
\end{center}

By using \eqref{eq:ngoes0} and the fact that the probability that the $n$-th particle,
if activated, ever visits the site $m>n$ is
\begin{equation}\label{eq:ngoesm}
  \left(\frac{1-\sqrt{1-4p_n^2l_n(1-l_n)}}{2p_nl_n}
\right)^{m-n},
\end{equation}
one gets a lower bound for the probability of local survival. 
The main idea in Theorem~\ref{excor:geomlocalsurvival} is that even labelled blocks 
take care of infinite activation while odd labelled blocks provide infinitely many particles
visiting 0. In particular we will suppose that 
 $\{\mathcal{B}_{2n}\}_{n \in \N}$ satisfies the conditions in Theorem~\ref{excor:geomglobalsurvival}  
(thus guaranteeing global survival).

\begin{teo}\label{excor:geomlocalsurvival}
Suppose that there exists $L\in \N$ and a sequence of pairwise disjoint sets $\{\mathcal{B}_n\}_{n \in \N}$
such that $\{\mathcal{B}_{2n}\}_{n \in \N}$ satisfies the hypotheses of Theorem~\ref{excor:geomglobalsurvival}.
Define $\Delta_k:=1-p_k$,  $\delta_k:=1/2-l_k$
and $\mathcal{O}:=\bigcup_{j \in \N} \mathcal{B}_{2j+1}$.  
%
If  one of the following holds:
\begin{enumerate}
\item
$\displaystyle \liminf_{k \to \infty, \, k \in \mathcal{O}} k \Big (2\delta_k+\sqrt{2\Delta_k+4\delta_k^2}\Big )<\infty$
\item
$\displaystyle 2\delta_k+\sqrt{2\Delta_k+4\delta_k^2} \le \log(k)/k$,
for all sufficiently large $k \in \mathcal{O}$
\item $\displaystyle 
\sum_{
{\scriptscriptstyle k \in  \mathcal{O}\colon 
\scriptscriptstyle 
p_k(1-l_k)>1/2
}}
  k^L \delta_k^{L} + \sum_{k \in  \mathcal{O}} k^L\Delta_k^{L/2}< +\infty$
\end{enumerate}
then there is local survival.
\end{teo}

Each of the conditions $(1),\,(2)$ and $(3)$ of Theorem~\ref{excor:geomlocalsurvival} implies the
divergence of
$\sum_{k \in \mathcal{O}}  
\big (1-\sqrt{1-4p_k^2l_k(1-l_k)} \big )^k/(2p_k(1-l_k))^k$ which, in turn, implies that, conditioned on infinite activation, 
the probability that an infinite number of particles in $\mathcal{E}$
visits the origin is positive.

Moreover, if we take $L=L_1$, then a sufficient condition for local (and global) survival is
$\sum_n n^L(1-p_n)^{L/2} < +\infty$, $\sum_{n\colon p_n l_n >1/2} (l_n-1/2)^L < +\infty$ and
$\sum_{n\colon p_n (1-l_n) >1/2} n^L(1/2-l_n)^L < +\infty$.

As a corollary we obtain the complete phase diagram in the case where $\Delta_k$ and $\delta_k$
decay polynomially (the analogous results in the immortal case have been studied in Examples~\ref{exmp:rightd}
and \ref{exmp:firstpossibility}).

\begin{cor} \label{cor:power}
Let $1-p_n \sim 1/n^\beta$ as $n \to \infty$ (where $\beta>0$).
\begin{enumerate}
\item If $l_n-1/2 \sim 1/n^\alpha$ as $n \to \infty$ (where $\alpha >0$)
 there is always global survival; moreover there is 
local survival if and only if $\beta \ge \min(2, 1+\alpha)$.
\item If $1/2-l_n\sim 1/n^\alpha$ as $n \to \infty$ (where $\alpha >0$)
there is always global survival; moreover there is 
local survival if and only if
$\beta \ge 2$ and $\alpha \ge 1$.
\end{enumerate}
\end{cor}

The phase diagram given by the previous corollary can be compared with the results of Examples~\ref{exmp:rightd} and
\ref{exmp:firstpossibility}, see the following figure.
\medskip

\begin{center}
 \includegraphics[width=8truecm]{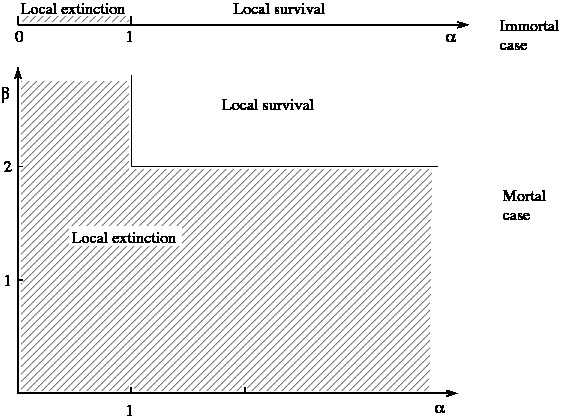} \hfill
\includegraphics[width=8truecm]{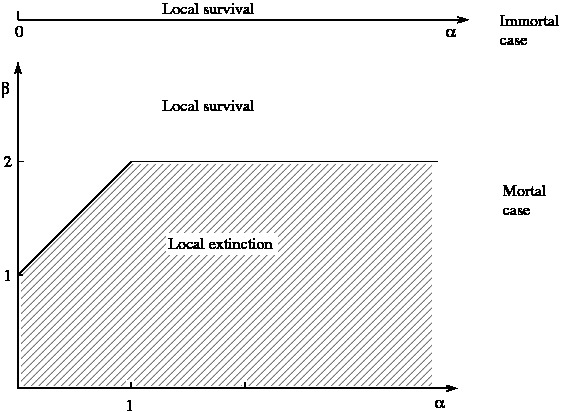} 
\\ \hfill {Figure 3. The right drift case.} \hfill \hfill {Figure 4. The left drift case.} \hfill \hfill \hfill
\end{center}

\section{Final Remarks}
\label{sec:final}

First of all we note that all our results in the immortal case apply 
to the inhomogeneous case where there is one particle at all vertices $\{n_k\}_{k \in \N}$
(where $\{n_k\}_{k \in \N}$ is a strictly increasing sequence in $\N$ such that $n_{k+1}-n_k\le m<+\infty$
for all $k \in \N$) and no particles 
elsewhere. The generalization is straightforward:
all sums and products run over the set of initially occupied vertices $\{n_k\colon  k \in \N\}$ instead of $\N$;
in particular Theorem~\ref{th:localsurv} holds in the inhomogeneous case for any generic
subset $\{n_k \colon k \in \N\}$ of occupied vertices (without restrictions).
On the other hand the results in the mortal case, since $p_n$ can be equal to $0$,  deal simultaneously
with the homogeneous case as well with the inhomogeneous case.


Let us discuss briefly the case where there are particles on the whole
 line $\mathbb Z$. When we say that the left (respectively right) process survives
globally  (respectively locally) we are talking about the process which involves just
the particle in the left (respectively right) side of the line (the origin is included).
Clearly if either the right or the left process survives (globally or locally)
then the whole process survives (globally or locally).
We consider the immortal particle case for simplicity 
and we sketch the differences with the mortal case.
Suppose that all the particles in the left (respectively~right) process
are activated then the 
probability of local survival is 1 or 0
depending on the divergence or convergence of the series
$\sum_{n} \min \big (1, \frac{l_n}{1-l_n} \big )$
(respectively~$\sum_{n} \min \big (1, \frac{1-l_n}{l_n} \big )$).

If the 
probabilities of local survival of both
the left and right processes are $0$, then there is local extinction for the whole
process as well. Indeed there might be cooperation between the particles in
two half lines in order to improve the activation process
but nothing can be done for local survival.
If either the left process or the right one can survive globally, then
there is a positive probability of local survival 
only if at least
one of the two process survives locally (once all the particle are activated).
%
%
Here we are not saying that one of the process survives locally by itself but
that it survives with positive probability once all its particles are activated (maybe by one
particle from the other side).
We observe that in the mortal particle case,
the local survival of the whole process is equivalent to the local survival
of one of the two half processes by itself.

If both processes cannot survive globally then there might still be global
survival; in order to survive globally it is sufficient (and necessary as
well) that an infinite number of particle from each 
side crosses the
origin and goes to the other
side.
In this case
global and local survival are equivalent.

Another question is what can be said in random environment, that is the case
where $\{l_n\}_{n \in \N}$ is a sequence of independent random variables
taking values in $(0,1)$ (also the sequence $\{p_n\}_{n \in \N}$ may
be randomly chosen).
The analysis of the random environment case exceeds the purpose
of this paper, nevertheless many results may be deduced from ours. Here is an example in the 
immortal particle case (with right drift).
A realization of the
\textit{environment} is a fixed realization $\{l_n(\omega)\}_{n \in \N}$.

\begin{teo}\label{th:localsurvRE}
Suppose that $\{l_n\}_{n \in \N}$ is a sequence of independent random variable such
that $\pr(l_n<1/2)=1$ for all $n \in \N$. Then either the probability of local survival
is $1$ for almost every realization of the environment or it is $0$ for almost every realization of the environment.\\
In particular we have the following sufficient conditions.
\begin{enumerate}
 \item If $\sum_{n \in \N} \pr(n(1/2-l_n)\le M)=+\infty$ for some $M$ then the probability of local survival
is $1$ for almost every realization of the environment.
\item If 
$\displaystyle \sum_{n \in \N} \pr \Big (\frac{n(1/2-l_n)}{\log(n)} >  \frac{1}{4- \log(n)/n} \Big )<+\infty $ 
then the probability of local survival
is $1$ for almost every realization of the environment.
\item If 
$\displaystyle \sum_{n \in \N} \pr \Big (\frac{n(1/2-l_n)}{\log(n)} <  
\frac{1+\beta \log(\log(n))/\log (n)}{4- 2(\log(n)+ \beta \log(\log(n)))/n} \Big ) <+\infty $ 
then the probability of local survival
is $0$ for almost every realization of the environment. 
\end{enumerate}
\end{teo}
This is a consequence of Theorem~\ref{th:localsurv}, Corollary~\ref{cor:condition1} and the Borel-Cantelli Lemma.
We note in particular that for conditions (2) and (3) we do not need the independence of $\{l_n\}_{n \in \N}$.

\section{Proofs}\label{sec:Proofs}
\addtocounter{section}{1}

\begin{proof}[Proof of Theorem \ref{th:localsurv}]
\begin{enumerate}
\item
Let $C_0$ be the event that the particle which starts at 0 visits all
vertices $n\ge1$: since $l_0<1/2$,
then $\Prob(B_\infty)=\Prob(C_0)=1$. Moreover, with respect to $\Prob(\cdot|C_0)$,
$\{A_n\}_{n\ge1}$ is an independent family of events;
$\Prob(A_n)=\Prob(A_n|B_n)=\Prob(A_n|C_0)$ for $n\ge1$.
Clearly in this case, $\{A_n\}_{n\ge1}$ is independent with respect to $\Prob$.
Thus
\[
\sum_{n\in \N}\Prob(A_n)=
\sum_{n\in \N}  \left(\frac{l_n}{1-l_n}\right)^n.
\]
The claim follows by the Borel-Cantelli Lemma.
\item
If all particles have a drift to the left, each particle visits 0 a.s.~only a finite number of
times. Hence in order to have local survival, we need to activate all particles. But infinite
activation is also a sufficient condition since starting at $n>0$ each particle visits 0
a.s.~at least once.
\end{enumerate}
\end{proof}

Before proving Corollary \ref{cor:condition1} we need a technical lemma. 

\begin{lem}\label{lem:sumexp}
Let $W=\sum_{i=1}^\infty (1-a_n)^n$ (where $a_n \in [0,1]$ for all $n \ge 1$) then
\begin{enumerate}
\item $\liminf_{n} na_n  < +\infty$ implies $W=+\infty$;
\item  if, for every sufficiently large $n$, $a_n  \le n^{-1}\log(n)$   then 
$W=+\infty$.
\item  if there exists $\beta > 1$ such that, for every sufficiently large $n$, $a_n  \ge n^{-1}\big (\log(n) + \beta \log(\log(n)) \big)$   then 
$W<+\infty$.

\end{enumerate}
\end{lem}

\begin{proof}
\begin{enumerate}
\item Let $\{n_j\}_{j \in \N}$ and $M>0$ such that $a_{n_j} \le M/n_j$ for all $j \in \N$. Hence,
$(1-a_{n_j})^{n_j} \ge (1-M/n_j)^{n_j} \to e^{-M}$ as $j \to \infty$.
\item 
Since, eventually as $n \to \infty$, $(1-a_n)^n \ge (1-q_n)^n$ where $q_n=\log(n)/n$, it is enough to prove that 
$W^\prime := \sum_{n \ge 1} (1-q_n)^n =+\infty$. To this aim, we prove that there exists $c>0$ such that
$ (1-q_n)^n \ge c/n$ for all $n \ge 1$. Indeed let $f(x):= x (1-\log(x)/x)^x$ (for all $x\ge 1$); then
\[
 \begin{split}
f(x)&= \exp \Big (x \log \Big (1- \frac{\log(x)}{x} \Big ) +\log(x) \Big ) \\
& = 
\exp \Big (x \Big (- \frac{\log(x)}{x}- \frac{\log(x)^2}{2 x^2} + O \Big ( \Big |\frac{\log(x)}{x} \Big |^3 \Big )\Big ) +\log(x) \Big ) \\
& =\exp \Big (-\frac{\log(x)^2}{2x} \Big (1 + O \Big ( \Big |\frac{\log(x)}{x}\Big | \Big ) \Big )\Big ) \to 1^-
 \end{split}
\]
as $x \to \infty$
(where, as usual, given two functions $h$ and $f$, by $h=O(f)$ as $x \to x_0$ we mean
$\limsup_{x \to x_0} |h(x)/f(x)| <\infty$). 
Hence, by continuity and compactness, since $f(x)>0$ for all $x \ge 1$, there is $c>0$ such that
$f(x) \ge c$ for all $x \ge 1$.
\item 
Take $\beta>1$ such that $a_n  \ge n^{-1}\big (\log(n) + \beta \log(\log(n)) \big)=:q_n$ for every sufficiently large $n$, then
\[
 \begin{split}
(1-a_n)^n & \le  (1- q_n )^n \le  \big ( (1- q_n )^{1/q_n} \big )^{nq_n} \le \exp{(-n q_n)} \\
&\le \exp{\big (-\log(n) - \beta \log(\log(n)) \big )}=\frac{1}{n \log(n)^\beta}
\end{split}
\]
hence $W<+\infty$.
\end{enumerate}
 
The previous proof implies in particular that if, for every sufficiently large $n$, $a_n  \le \log(n) /n$   then 
$\sum_{i \in J}^\infty (1-a_n)^n=+\infty$ for every $J \subseteq \N \setminus \{0\}$ such that $\sum_{i \in J}^\infty 1/n=+\infty$.

\begin{proof}[Proof of Corollary \ref{cor:condition1}]
Rewrite $l_n/(1-l_n)$ as $1-a_n$ where  $a_n:= (1-2l_n)/(1-l_n)$.
\begin{enumerate}[(1)]
 \item  We note that $\liminf_n n(1/2-l_n)< \infty$ is equivalent to
$\liminf_n n a_n< \infty$ (since in both cases $\lim_{n \to \infty} l_n =1/2$),
thus according to Lemma~\ref{lem:sumexp}(1)
$\sum_{n=1}^\infty  \left( l_n/(1-l_n) \right)^n$ diverges.
Theorem~\ref{th:localsurv} yields the result.
 \item 
Since, $n(1/2-l_n)/\log(n)\le  {1}/{(4- \log(n)/n)}$ is equivalent to
$a_n \le  \log(n) /n$, then, 
 according to Lemma~\ref{lem:sumexp}(2),
$\sum_{n=1}^\infty  \left( l_n/(1-l_n) \right)^n$ diverges and
Theorem~\ref{th:localsurv} yields the result.
\item  
 We note that the inequality 
 $n(1/2-l_n)/\log(n)\ge  (1+\beta \log(\log(n))/\log(n))/(4- 2(\log(n)+ \beta \log(\log(n)))/n)$
is equivalent to
 $
a_n \ge \log(n)/n+\beta \log(\log(n))/n$, thus, 
  according to Lemma~\ref{lem:sumexp}(3), the series
$\sum_{n=1}^\infty  \left( l_n/(1-l_n) \right)^n$ converges and
Theorem~\ref{th:localsurv} yields the result.
\item Note that $(1-\frac{1-2l_{n}}{1-l_{n}})^{\frac{1-l_{n}}{1-2l_{n}}} \le {1/e}$. Then
\[
\left (\frac{l_{n}}{1-l_{n}} \right )^{n} \le
\exp \left (-\frac{2n (1/2-l_{n})}{1-l_{n}} \right ).
\]
We divide the sum into two 
series
\[
\sum_n \left (\frac{l_{n}}{1-l_{n}}\right )^{n} \le
\sum_{n: {2/(1-l_n)} \le \lambda} \left (\frac{l_{n}}{1-l_{n}}\right )^{n} +
\sum_{n: {2/(1-l_n)} > \lambda} \exp\left  (-\frac{2n (1/2-l_{n})}{1-l_{n}}\right ).
\]
The second series is finite since each summand  is
bounded from above by $\exp(-\lambda n ({1/2}-l_n))$.
The first series is finite since ${2/(1-l_n)} \le \lambda$
implies $l_n/(1-l_n) \le \lambda/2 -1 <1$.
Again, Theorem~\ref{th:localsurv} yields the conclusion.
\end{enumerate}

\end{proof}

\begin{proof}[Proof of Proposition~\ref{pro:rumor}]
The proof is based on a comparison between our frog model and the heterogeneous firework process introduced
in  \cite{cf:JMZ}.
Consider the explicit construction of the frog model given in Section~\ref{sec:intro}. 
Define a family of independent random variables $\{R_i\}_{i \in \N}$ as
$R_i := \max\{Z^i_n \colon n \in \N \}$, that is, the maximum right excursion of the independent walker
$\{Z^i_n\}_{n \in \N}$. Conditioned on the activation of the $i$-th walker, that is $T_i<\infty$, 
$R_i$ is also the maximum right excursion of the
dependent walker $\{X^i_n\}_{n \in \N}$ of our frog model.
Let $\{R_i\}_{i \in \N}$ be the radii of the firework process.
In the firework process sites are activated as follows:
at time $0$, site $0$ sends a signal and activates all the sites to its right up to a distance $R_0$. 
Iteratively, if site $i$ is activated, it sends a signal which will activate all the inactive sites, if any,
in the interval $(i, i+R_i]$.
We prove that site $i$ is activated in the firework process if and only if the $i$-th walker is activated
in our frog model. When $i=0$ there is nothing to prove. Suppose it holds for all $i \le N$.
The site $N+1$ is activated if and only if all the sites between $0$ and $N$ are activated and $N+1 \le R_i+i$
for some $i \le N$. By induction, this is equivalent to the event ``all the walkers starting between $0$ and $N$ are activated
and, for some $i \le N$, the $i$-th walker reaches $N+1$'', that is, the $N+1$-th walker is activated.

Following the proof of  \cite[Proposition 3.3]{cf:JMZ} (in that paper $B_n$ is denoted by $V_n$), we note that 
\[B_{n_{k+1}} \supseteq B_{n_k} \cap \Big ( \bigcup_{i=0}^{n_k} (R_i \ge n_{k+1}-i) \Big ).\] 
By basic random walk theory $\pr(R_i \ge k)=(1-l_i)^k/l_i^k$.
Using the FKG inequality and the independence of the  $\{R_i\}_{i \in \N}$ we have that
\[
 \pr(B_{n_{k+1}}) \ge \pr(B_{n_{k}}) \Big ( 1- \prod_{i=0}^{n_k} \Big (
1 - \Big ( \frac{1-l_i}{l_i}\Big )^{n_{k+1}-i}
\Big )\Big )
\]
hence the probability of  survival  of the firework process (and of our frog model as well) is bounded from below by
\[
 \prod_{k=0}^\infty 
\Big ( 1- \prod_{i=0}^{n_k} \Big (
1 - \Big ( \frac{1-l_i}{l_i}\Big )^{n_{k+1}-i}
\Big )\Big ).
\]
According to Lemma~\ref{lem:test0}(2), the previous product is strictly positive if and only if
\[
  \sum_{k=0}^\infty  \prod_{i=0}^{n_k} \Big (
1 - \Big ( \frac{1-l_i}{l_i}\Big )^{n_{k+1}-i}
\Big ) < +\infty.
\]
\end{proof}

\begin{proof}[Proof of Theorem~\ref{pro:blocks}]
 The proof can be easily adapted from the proof of Theorem~\ref{excor:geomlocalsurvival}.
\end{proof}

\begin{proof}[Proof of Proposition \ref{pro:RWapproach}]
Note that, since $l_n> 1/2$ for all $n \in \mathbb{N}$, then
$\liminf_n l_n > 1/2$
is equivalent to $\inf_n l_n > 1/2$.
We associate to the process a random walk on a subset of $\N \times \N$.
To this aim we define the generation $0$ as the set containing only the initial active particle and, recursively,
the generation $n+1$ as the set of vertices visited by at least one particle of generation $n$.
We denote by $j_{n+1}$ the rightmost position reached by a particle of a generation $i \le n$.
Hence the generation $n$ is nonempty if and only if $j_n >j_{n-1}$, in this case it contains
all the particles starting in the set of positions $\{j_{n-1}+1, \ldots, j_n\}$.
It is clear that if the $n$-th generation is empty then all generations $m \ge n$ are empty as well.
The system survives locally if and only if all the particles are activated, that is,
if and only if every generation contains at least one particle.

As a warm-up we
 start with the simpler case of an homogeneous system: $l_n=l> 1/2$ for every $n$.
We associate to this process the random walk $\{\Delta_n\}_n$ which counts the particles of the
generation $n$, which is $\Delta_n=j_n-j_{n-1}$. The origin is the only absorbing state of this
Markov chain. It is easy to compute the probability of absorption (or local extinction)
\[
\begin{split}
 \Prob(\Delta_n=0 | \Delta_{n-1}=h) &= \left ( 1 - \frac{1-l}{l} \right )\left ( 1 - \left (\frac{1-l}{l} \right )^2 \right )
\cdots \left ( 1 - \left ( \frac{1-l}{l} \right )^h \right ) \\
 & \ge \prod_{i=1}^\infty \left ( 1 - \left (\frac{1-l}{l} \right )^i \right )
\end{split}
\]
which is strictly positive according to Lemma~\ref{lem:test0}.
This implies, in particular, that the Markov chain $\{\Delta_n\}_n$ is absorbed
in $0$ a.s., whence $\Prob(B_\infty)=0$.

In the general case of an inhomogeneous system, $\{\Delta_n\}_n$ is no longer
a Markov process. In order to be able to mimic the steps above, we must consider
the Markov chain $\{(\Delta_n, j_n)\}_n$.
In this case
\[
\begin{split}
 \Prob \big (\Delta_n=0 | (\Delta_{n-1}, j_{n-1})=(h,k) \big )
& = \prod_{i=k-h+1}^k \left ( 1 - \left (\frac{1-l_i}{l_i} \right )^{k-i+1} \right ) \\
& \ge \inf_{h,k \in \N\colon  h \le k} \prod_{i=k-h+1}^k \left ( 1 - \left (\frac{1-l_i}{l_i} \right )^{k-i+1} \right ) \\
& = \inf_{k \in \N} \prod_{i=1}^k \left ( 1 - \left (\frac{1-l_i}{l_i} \right )^{k-i+1} \right ).
\end{split}
\]
%
%
Note that $\inf_{k \in \N} \prod_{i=1}^k \left ( 1 - \left (\frac{1-l_i}{l_i} \right )^{k-i+1} \right )>0$ is
equivalent to $\inf_{i \in \N} l_i > 1/2$ and implies $\Prob(B_\infty)=0$.
%
%
\end{proof}

\begin{proof}[Proof of Proposition~\ref{excor:geomglobalextinction}]
Suppose that $\sup_n p_n=p<1$ and that $n$ dormient particles are activated in $n$ consecutive vertices, say $i, i+1, \ldots, i+n-1$.
The probability that the lifespan of all these particles is so short that neither of them can possibly reach
the vertex $i+n$ (and activate more particles) is
\[
 \prod_{j=0}^{n-1} (1-p_{i+j}^{n-j}) \ge \prod_{j=1}^{\infty} (1-p^{j})>0, \qquad \forall n \in \N.
\]
As in 
the proof of Proposition~\ref{pro:RWapproach},
$\pr \big (\Delta_n=0 | (\Delta_{n-1},j_{n-1})=(n,i+n-1) \big ) \ge \prod_{j=1}^{\infty} (1-p^{j})>0$
and the conclusion follows.
\end{proof}

\begin{proof}[Proof of Theorem~\ref{excor:geomglobalsurvival}]
Define $L_1:= \sup_{n \in \N}(\max(\mathcal{B}_{n+1})-\min(\mathcal{B}_{n}))$.
We note that the series 
$\sum_{n \in \N} \sum_{k \in \mathcal{B}_n} \sqrt{1-p_n}$ 
is convergent, hence 
we have that
in all but a finite number (say $N_0$) 
of blocks at least one particle has a strictly positive lifetime parameter $p_n$.
Since there is always a positive probability that the particle at $0$ reaches
$\max(\mathcal{B}_{N_0+1})$, then we can assume without loss of generality that in every block $\mathcal{B}_n$ there is
at least one particle with strictly positive lifetime parameter.

Consider the (mortal) random walk with $p(j,j-1)=p_nl_n$, $p(j,j+1)=p_n(1-l_n)$,
$p(j,D)=1-p_n$ for all $j \in \Z$, $p(D,D)=1$ ($D$ represents the absorbing state where the particle is 
considered dead).
Define
\[
f_n^{(k)}(x,y)=\Prob(\text{the $n$-th RW visits }y\text{ for the first time at time }k+h|\text{the RW is at }x \text{ at time } h).
\]
Let $F_n(x,y|z)=\sum_k f_n^{(k)}(x,y)z^k$. Then
\begin{equation}\label{eq:recurrenceforF}
F_n(x-1,x|z)=p_n (1-l_n) z+p_nl_nzF_n(x-1,x+1|z).
\end{equation}
Noting that $F_n(x-1,x+1|z)=(F_n(x-1,x|z))^2$ we get 
\begin{equation}\label{eq:1stepright}
F_n(x-1,x|z)=\frac{1-\sqrt{1-4z^2p_n^2l_n(1-l_n)}}{2zp_nl_n}=\frac{2zp_{n} (1-l_{n})}{1+\sqrt{1-4z^2p_{n}^2l_{n}(1-l_{n})}}.
\end{equation}
Hence the probability for a mortal particle starting from $j_n$ to ever reach $j_{n+1}$
is $F_{j_n}(x-1,x|1)^{j_{n+1}-j_n}$.

The probability that,
in each block $\mathcal{B}_n$ such that $\max(\mathcal{B}_{n+1})>\max(\mathcal{B}_n)$, 
there exists at least one particle which visits all the sites of the following
block is bounded from below by the probability that in every block at least one particle has a right excursion larger than
$L_1$, that is
 \begin{equation}\label{eq:product0}
 \prod_{n=1}^\infty
\left ( 1-\prod_{k \in \mathcal{B}_n}
\left ( 1-
\left(\frac{1-\sqrt{1-4p_k^2l_k(1-l_k)}}{2p_kl_k}
\right)^{L_1} \right )
\right).
\end{equation}

By Lemma~\ref{lem:test0} a sufficient condition for the positivity of the product in equation~\eqref{eq:product0}
is
\begin{equation}\label{eq:suffcondprod0}
\sum_{n \in \N}
\prod_{k \in \mathcal{B}_n}
\left ( 1-
\left(\frac{1-\sqrt{1-4p_k^2l_k(1-l_k)}}{2p_kl_k}
\right)^{L_1} \right )
< +\infty;
\end{equation}
the fact that in each block there is at least one particle, say at $k$, with $p_k >0$
implies that each term in the infinite product \eqref{eq:product0} is strictly positive and
Lemma~\ref{lem:test0} applies.
Since $1-x^n \le n(1-x)$ (for all $n \in \N$) and by using the following estimates
\[
\begin{split}
 0 &\le 1-
\frac{2p_{k} (1-l_{k})}{1+\sqrt{1-4p_{k}^2l_{k}(1-l_{k})}}
= \frac{1+\sqrt{(2p_kl_k-1)^2 +4p_{k}(1-p_k)l_{k}}-2p_{k} (1-l_{k})}{1+\sqrt{1-4p_{k}^2l_{k}(1-l_{k})}} \\
& \le 1-2p_k(1-l_k)+2\sqrt{1-p_k} + |2p_kl_k-1|
\le W_k:=
\begin{cases}
\displaystyle 2(1-p_k)+2\sqrt{1-p_k} & \text{if } p_k l_k \le 1/2 \\
\displaystyle 4p_k(l_k-1/2)+ 2\sqrt{1-p_k}  & \text{if } p_k l_k > 1/2, \\
\end{cases} \\
\end{split}
\]
we have that equation~\eqref{eq:suffcondprod0} is implied by
 $\sum_{n \in \N} \prod_{k \in \mathcal{B}_n} W_k <+\infty$
which, in turn, is implied by
$\sum_{n \in \N}  \prod_{k \in \mathcal{B}_n} S_k <+\infty$
where
\[
  S_k :=
\begin{cases}
 \sqrt{1-p_k} & \text{if }p_kl_k \le 1/2 \\
\sqrt{1-p_k} + 2p_k(l_k-1/2) & \text{if }p_kl_k > 1/2 \\
\end{cases}
\]
since 
$ \prod_{k \in \mathcal{B}_n} W_k \le 4^L  \prod_{k \in \mathcal{B}_n} S_k$.

By using the inequality between arithmetic and geometric means we have that
%
%
%
%
$\sum_{n\in \N} \prod_{k \in \mathcal{B}_n} S_k \le \frac{1}{L}
\sum_{n \in \N} \sum_{k \in \mathcal{B}_n}  S_{k}^L$.
Hence $\sum_{n \in \N} \sum_{k \in \mathcal{B}_n} S_k^L<+\infty$
implies global survival.
Using, on the one hand, the Minkowski inequality and, on the other, the fact that $S_k$ is the sum
of the two nonnegative functions
$\sqrt{1-p_k}$ and $2p_k(l_k-1/2) \ident_{(0,+\infty)}(l_k-1/2)$ , we have that
$\sum_{n \in \N} \sum_{k \in \mathcal{B}_n} S_k^L<+\infty$ is equivalent to
$\sum_{n \in \N} \sum_{k \in \mathcal{B}_n} (1-p_k)^{L/2} < +\infty$, $\sum_{n=0}^\infty 
\sum_{k \in \mathcal{B}_n\colon p_k l_k >1/2} (l_k-1/2)^L < +\infty$
(since,  in both cases, $p_k \to 1$ as $k \to \infty$, $k \in \bigcup_{n \in \N}\mathcal{B}_n$).
\end{proof}

%

\begin{proof}[Proof of Theorem~\ref{pro:geomlocalextinction}]
We note that in this case, by switching $l_n$ and $1-l_n$ in equation~\eqref{eq:1stepright}
\begin{equation}\label{eq:1stepleft}
\Prob(A_n|B_n)=\left(
\frac{1-\sqrt{1-4p_n^2l_n(1-l_n)}}{2p_n(1-l_n)}
\right)^n
=\left( \frac{2p_n l_n}{1+\sqrt{1-4p_n^2l_n(1-l_n)}} \right)^n.
\end{equation}
Now, since $A_n\subset B_n$, $\Prob(A_n)\le\Prob(A_n|B_n)$ and
and by the Borel-Cantelli Lemma we have that $\sum_n \Prob(A_n|B_n)<+\infty$ implies
$\Prob(A_i\text{ i.o.})=0$.

\begin{enumerate}
\item 
Note that it is enough to prove the result in the case $\Delta_n \to 0$ and $\delta_n \to 0$. The result in the general case follows 
from a coupling between the process and a similar one with $p^\prime_n=1-\Delta^\prime_n$ and $l^\prime_n=1-\delta_n^\prime$
such that 
$\Delta_n \ge \Delta^\prime_n \to 0$, $\delta^\prime_n=\delta_n$ if $\delta_n <0$, 
$0 \le \delta^\prime_n \le \delta_n$ if $\delta_n\ge 0$,
$\delta_n^\prime \to 0$
 and $\liminf_{n\to \infty} n \big (2\delta^\prime_n+\sqrt{2\Delta^\prime_n+4{\delta_n^\prime}^2} \big )/\log(n)>1$. 
Hence from now on we suppose that $\Delta_n \to 0$ and $\delta_n \to 0$ as $n \to \infty$.

In order to check whether  the series $\sum_n \Prob(A_n|B_n)$ is convergent, we use 
 Lemma~\ref{lem:sumexp}, hence it is important to estimate $a_n=1-(1-\sqrt{1-4p_n^2l_n(1-l_n)})/(2p_n(1-l_n))$. 
To this aim we write
 \begin{equation}\label{eq:asym}
\begin{split} 
\pr(A_n|B_n)=
&=\frac{(1-\Delta_n)(1-2 \delta_n)}{1+\sqrt{1-(1-\Delta_n)^2 (1-4\delta_n^2)}}
 \\
 &= (1-\Delta_n)(1-2 \delta_n) \Big (1-\sum_{j=1}^\infty (-1)^{j+1} \big (1-(1-\Delta_n)^2 (1-4\delta_n^2) \big )^{j/2} \Big )\\
&= (1-\Delta_n)(1-2 \delta_n) \Big (1-\sqrt{2\Delta_n+4\delta_n^2} +o \Big (\sqrt{\Delta_n+2\delta_n^2} \Big ) \Big )\\
&=1-(2\delta_n+\sqrt{2\Delta_n+4\delta_n^2}) +
 o \Big (\sqrt{2\Delta_n+4\delta_n^2}\Big ) \\
\end{split}
 \end{equation}
as $\Delta_n \to 0$ and $\delta_n \to 0$
(where, as usual, given two functions $h$ and $f$, by $h=o(f)$ as $x \to x_0$ we mean that $\lim_{x \to x_0} h(x)/f(x)=0$).
If $\liminf_{n \to \infty} n \big (2\delta_n+\sqrt{2\Delta_n+4\delta_n^2} \big )/\log(n)>1$, 
then $\liminf_{n \to \infty} n \big (2\delta_n+\sqrt{2\Delta_n+4\delta_n^2}+
 o \big (\sqrt{2\Delta_n+4\delta_n^2}\big ) \big )/\log(n)>1$.
Thus, according to Lemma~\ref{lem:sumexp}(3), the series $\sum_n \Prob(A_n|B_n)$ converges.

\item
Given the following estimate 
\[
\Prob(A_n|B_n) = \frac{2p_n l_n}{1+\sqrt{1-4p_n^2l_n(1-l_n)}}  \le
\min( p_n, 2p_nl_n)=p_n(1-(1-2l_n)^+),
\]
we have that $\sum_n p_n^n(1-(1-2l_n)^+)^n<+\infty$ implies $\sum_n \Prob(A_n|B_n)<+\infty$.
\item
It follows from $\Prob(A_n|B_n) \le p_n$ for all $n \in \N$.

\end{enumerate}

\end{proof}

\begin{proof}[Proof of Theorem~\ref{excor:geomlocalsurvival}]
Consider the following conditions:
\begin{equation}\label{eq:better1}
\begin{cases}
\displaystyle  
\sum_{{\scriptscriptstyle k \in \bigcup_{n \in \N}  \mathcal{B}_{2n}\colon \scriptscriptstyle p_{k} l_{k} >1/2}}
(l_{k}-1/2)^{L} +\sum_{k \in  \bigcup_{n \in \N}  \mathcal{B}_{2n}} (1-p_{k})^{L/2}  < +\infty \\
\sum_{k \in \mathcal{O}}  \left(\frac{1-\sqrt{1-4p_k^2l_k(1-l_k)}}{2p_k(1-l_k)}
\right)^k=+\infty \\
\end{cases}
\end{equation}
The first line of the previous equation, according to Theorem~\ref{excor:geomglobalsurvival},
implies global survival and the activation of every particle. 
The divergence of the series in the second line implies that, once activated, an infinite number of particles
visits the origin (whence local survival).

We are left to prove that each of the three conditions implies that 
the series in the second line of the previous 
equation is divergent.

\begin{enumerate}
 \item 
Given a sequence $\{k_j\}_{j \in \N}$ in $\mathcal{O}$ such that 
$\lim_{j \to \infty} k_j \Big (2\delta_{k_j}+\sqrt{2\Delta_{k_j}+4\delta_{k_j}^2}\Big )<\infty$, then
$\Delta_{k_j} \to 0$ and $\delta_{k_j} \to 0$ as $j \to \infty$. Elementary but tedious computations show that,
the asymptotic estimate  of
$a_{k_j}=1-(1-\sqrt{1-4p_{k_j}^2l_{k_j}(1-l_{k_j})})/(2p_{k_j}(1-l_{k_j}))$ 
given 
by equation~\eqref{eq:asym}, can be refined obtaining
\[
\begin{split} 
\frac{1-\sqrt{1-4p_{k_j}^2l_{k_j}(1-l_{k_j})}}{2p_{k_j}(1-l_{k_j})} 
& \ge 1-(2\delta_{k_j}+\sqrt{2\Delta_{k_j}+4\delta_{k_j}^2}).
\end{split}
 \]
Thus $\lim_{j \to \infty} k_j a_{k_j}<\infty$
hence, according to Lemma~\ref{lem:sumexp}(1),
the series in the second line of equation~\eqref{eq:better1} is divergent.
\item
In this case $\Delta_k \to 0$ and $\delta_k \to 0$ as $k \to \infty$, $k \in \mathcal{O}$.
The conclusion follows from Lemma~\ref{lem:sumexp}(2) as before, since
$ 
\displaystyle a_k  \le 2\delta_k+\sqrt{2\Delta_k+4\delta_k^2} \le \log(k)/k$,
for all sufficiently large $k \in \mathcal{O}$.
To be precise, since the sum in the second line of equation~\eqref{eq:better1} 
spans in $\mathcal{O}$ we cannot  apply
Lemma~\ref{lem:sumexp}(2) as it is. Nevertheless the same argument of the proof of Lemma~\ref{lem:sumexp}(2) 
holds in this case as well; in particular, under our hypotheses, the $k$-th summand of the series in the second
line of equation~\eqref{eq:better1}
is
bounded from below by $c/k$ for some $c>0$ and, since $\sup_{n \in \N} (\max(\mathcal{B}_{2n+2})-\min(\mathcal{B}_{2n}))<+\infty$,
$\sum_{k \in \mathcal{O}} 1/k =+\infty$ (see the discussion after the proof of the Lemma).

\item
As in the proof of Theorem~\ref{excor:geomglobalsurvival}
we assume, without loss of generality, that in every block $\mathcal{B}_n$ there is
at least one particle with strictly positive lifetime parameter.

If we prove that 
\begin{equation}\label{eq:suffcondprod}
\displaystyle
\sum_{n \in \N} \prod_{k\in \mathcal{B}_{2n+1}}
\left ( 1-
\left(\frac{1-\sqrt{1-4p_k^2l_k(1-l_k)}}{2p_k(1-l_k)}
\right)^{k} \right ) < +\infty
\end{equation}
then we have that
\[
 \prod_{k\in \mathcal{O}}
\left ( 1-
\left(\frac{1-\sqrt{1-4p_k^2l_k(1-l_k)}}{2p_k(1-l_k)}
\right)^{k} \right ) =0,
\]
which, by Lemma~\ref{lem:test0}, is equivalent to the second line of equation~\eqref{eq:better1}.

Since $(1-x)^k \le k(1-x)$,
by using the following estimate
\[
  0 \le 1-
\frac{2p_{n} l_{n}}{1+\sqrt{1-4p_{n}^2l_{n}(1-l_{n})}}
\le \widetilde W_n:=
\begin{cases}
\displaystyle 2(1-p_n)+2\sqrt{1-p_n} & \text{if } p_n (1-l_n) \le 1/2 \\
\displaystyle 4p_n(1/2-l_n)+ 2\sqrt{1-p_n}  & \text{if } p_n (1-l_n) > 1/2. \\
\end{cases} 
\]
we have that equation~\eqref{eq:suffcondprod} is implied by
$\sum_{n \in \N} \prod_{k\in \mathcal{B}_{2n+1}} k \widetilde W_k <+\infty$
which, in turn, is implied by
$\sum_{n \in \N} \prod_{k\in \mathcal{B}_{2n+1}} k \widetilde S_k <+\infty$
where
\[
\widetilde S_n:=
\begin{cases}
\displaystyle \sqrt{1-p_n} & \text{if } p_n (1-l_n) \le 1/2 \\
\displaystyle 2p_n(1/2-l_n)+ \sqrt{1-p_n}  & \text{if } p_n (1-l_n) > 1/2. \\
\end{cases} 
\]
since $\prod_{k\in \mathcal{B}_{2n+1}} k\widetilde W_k \le 4^L \prod_{k\in \mathcal{B}_{2n+1}} k\widetilde S_k$.
The conclusion follows using the inequality between arithmetic and geometric means and the Minkowski inequality 
as in the proof of Theorem~\ref{excor:geomglobalsurvival}.

\end{enumerate}
\end{proof}

\begin{proof}[Proof of Corollary~\ref{cor:power}]

If $L> \max(2/\beta, 1/\alpha)$ then $\sum_{n \in \N}(\Delta^{L/2}+|\delta_n|^L) <\infty$ hence, 
according to Theorem~\ref{excor:geomglobalsurvival}, for all $\alpha, \beta >0$ there is
global survival in both cases.

Define $r_n:= 2\delta_n +\sqrt{2\Delta_n+4\delta_n^2} \equiv  2\Delta_n/(\sqrt{2\Delta_n+4\delta_n^2}-2\delta_n)$,
where, using the same notation of Theorems~\ref{pro:geomlocalextinction} and~\ref{excor:geomlocalsurvival},
$p_n=1-\Delta_n$ and $l_n=1/2-\delta_n$.
\begin{enumerate}
 \item 
 Suppose that $\Delta_n\sim 1/n^\beta$ and $\delta_n \sim -1/n^\alpha$.
Elementary computations show that, as $n \to \infty$,
\[
 r_n \sim
\begin{cases}
C_{\alpha,\beta} \sqrt{\Delta_n} \sim \frac{C_{\alpha,\beta}}{n^{\beta/2}} & \textrm{if } \beta \le 2\alpha \\
 C_{\alpha,\beta} \frac{\Delta_n}{\delta_n} \sim \frac{C_{\alpha,\beta}}{n^{\beta -\alpha}} & \textrm{if } \beta 
> 2\alpha \\
\end{cases}
\quad
\textrm{where}
\quad
C_{\alpha,\beta}=
\begin{cases}
\sqrt{2} & \textrm{if } \beta < 2\alpha \\
\sqrt{6}-2 & \textrm{if } \beta = 2\alpha \\
\frac12 & \textrm{if } \beta > 2\alpha \\
\end{cases}
\]
thus
\[
 \lim_{n \to \infty}\frac{n r_n}{\log(n)}
=
\begin{cases}
0 & \textrm{if } \{\beta \le 2\alpha, \beta \ge 2\} \cup \{\beta 
>2 \alpha, \beta \ge 1+\alpha\} \equiv \{\beta \ge \min(2, 1+\alpha)\} \\
+\infty &\textrm{if }   \{\beta \le 2\alpha, \beta < 2\} \cup \{\beta 
>2 \alpha, \beta < 1+\alpha\}  \equiv \{\beta < \min(2, 1+\alpha)\} \\ 
\end{cases}
\]
\item
Suppose that $\Delta_n\sim 1/n^\beta$ and $\delta_n \sim 1/n^\alpha$.
Since
$ \max(4\delta_n,\sqrt{2\Delta_n}) \le r_n \le 4\delta_n+\sqrt{2\Delta_n}$
we have that, for every $\alpha, \beta>0$ there exists $\varepsilon \in (0,1)$ such that 
\[
(1-\varepsilon) \frac{n}{\log(n)} 
\max \Big (\frac{4}{n^\alpha},\sqrt{\frac2{n^\beta}}\Big ) \le \frac{n r_n}{\log(n)} \le 
 (1+\varepsilon)\frac{n}{\log(n)}\Big (\frac{4}{n^\alpha}+\sqrt{\frac2{n^\beta}} \Big) 
\]
for every sufficiently large $n\in \N$.
Whence
\[
 \lim_{n \to \infty}\frac{n r_n}{\log(n)}
=
\begin{cases}
0 & \textrm{if } \{\beta \ge 2, \alpha \ge 1\} \\
+\infty &\textrm{if }   \{\beta < 2\} \cup \{\alpha <1\}. \\ 
\end{cases}
\]
The conclusion follows as before.
\end{enumerate}

\end{proof}

%
%

\begin{lem}
\label{lem:test0}
 Let $\{\alpha_i\}_{i \in \N}$ and $\{k_i\}_{i\in \N}$ be such that $\alpha_i \in (-\infty, 1)$ and
$k_i \ge 0$ for all $i \in \N$.
\begin{enumerate}
 \item \[
        \sum_{i \in \N} k_i \alpha_i < +\infty \Longleftarrow \prod_{i \in \N} (1-\alpha_i)^{k_i}>0;
       \]
 \item moreover if $\alpha_i \in [0, 1)$ and $k_i \ge 1$ eventually as $i \to \infty$ then
\[
        \sum_{i \in \N} k_i \alpha_i < +\infty \Longleftrightarrow \prod_{i \in \N} (1-\alpha_i)^{k_i}>0;
       \]
\item If $\alpha_i(j) \in [0, 1-\epsilon]$ (for some $\epsilon>0$) and $k_i(j) \ge 1$ for all $i,j \in \N$ then
\[
        \sup_{j \in \N} \sum_{i \in \N} k_i(j) \alpha_i(j) < +\infty \Longleftrightarrow \inf_{j \in \N}
\prod_{i \in \N} (1-\alpha_i(j))^{k_i(j)}>0.
       \]
\end{enumerate}
\end{lem}

\begin{proof} Clearly $\prod_{i \in \N} (1-\alpha_i)^{k_i}>0$ if and only if $\sum_{i \in \N} k_i \log(1-\alpha_i)>-\infty$.
\begin{enumerate}
 \item
Observe that $\log(1-x) \le -x$ for all $x < 1$ hence
\begin{equation}
\label{eq:test0}
  \sum_{i \in \N} k_i \alpha_i \le -\sum_{i \in \N} k_i \log(1-\alpha_i)<\infty.
\end{equation}
 \item
In this case, since $k_i \ge 1$ both sides imply $\alpha_i \to 0$. Thus $\log(1-\alpha_i) \sim -\alpha_i$ and
\[
 \sum_{i \in \N} k_i \log(1-\alpha_i) > -\infty \Longleftrightarrow \sum_{i \in \N} k_i \alpha_i < \infty.
\]
\item
If $\inf_{j \in \N} \prod_{i \in \N} (1-\alpha_i(j))^{k_i(j)}>0$ then using the first inequality in equation~\eqref{eq:test0}
we obtain $\sup_{j \in \N} \sum_{i \in \N} k_i(j) \alpha_i(j) < +\infty$.
Conversely, it suffices to note that there exists $\delta \in(0,1)$ such that $-\delta \alpha_i(j) \le \log(1-\alpha_i)$.
\end{enumerate}

%
%
%
%
%

\end{proof}

\begin{lem}\label{lem:test}
Let $\{\alpha_i\}_i$ be a sequence of nonnegative numbers. Define
$\bar \alpha_n:=\min\{\alpha_i\colon i \le n\}$;
the following are equivalent:
\begin{enumerate}
\item
there exists
an increasing sequence $\{n_i\}$ such that
  $\sum_{i}(n_{i+1}-n_i)\alpha_{n_i} < +\infty$;
\item
either $\alpha_n=0$ for infinitely many $n \in\N$ or it is possible to define recursively
an infinite, increasing sequence $\{r_j\}_j$ by
\begin{equation}\label{eq:test}
 \begin{cases}
  r_0 :=\min \{n\colon \alpha_m>0, \, \forall m \ge n\} \ \\
r_{n+1}=\min \{i>r_n\colon  \alpha_i < \alpha_{r_n} \}
 \end{cases}
\end{equation}
and
$\sum_{i}(r_{i+1}-r_i)\alpha_{r_i} < +\infty$;
\end{enumerate}
Moreover if $\alpha_i>0$ for all $i \in \N$ then the previous assertions are equivalent to
\begin{itemize}
\item[\textit{(3)}] $\sum_{i} \bar \alpha_{i} < +\infty$.
\end{itemize}

\end{lem}

%
%
%
%

\begin{proof}

$(1) \Longrightarrow (2)$.
Suppose that  $(1)$  holds and there exists $r_0 \in \N$ such that, for all $n \ge r_0$ we have
$\alpha_n>0$. Suppose that $r_0=0$ (the proof in the general case follows easily from this particular case).
Then $\lim_i \alpha_{n_i}= 0$ and, since $\alpha_{n_i}>0$ for all $i \in \N$, we have that
for all $j \in \N$ the set $\{i\colon  \alpha_{n_i} < \alpha_j\} \not = \emptyset$  and it is possible to define recursively
the sequence $\{r_n\}$.
Clearly we have
\begin{equation}\label{eq:test2}
 \alpha_i \ge \alpha_{r_n}, \qquad \forall  i <r_{n+1}.
\end{equation}
We show now that for all increasing sequences $\{n_i\}_i$ we have
 \[
  \sum_{i}(n_{i+1}-n_i)\alpha_{n_i} \ge
  \sum_{i}(r_{i+1}-r_i)\alpha_{r_i}
 \]
which implies easily $(2)$.
Indeed, note that if we define $\gamma_j=\alpha_{r_i}$ for all $j \in [r_i, r_{i+1})$
then
\begin{equation}\label{eq:test3}
 \sum_{i}(r_{i+1}-r_i)\alpha_{r_i}= \sum_j \gamma_j;
\end{equation}
similarly
if $\gamma^\prime_j=\alpha_{n_i}$ for all $j \in [n_i, n_{i+1})$
then
\[
 \sum_{i}(n_{i+1}-n_i)\alpha_{n_i}= \sum_j \gamma^\prime_j.
\]
Let us fix $j \in \mathbb N$ and suppose that $j \in [r_i, r_{i+1}) \cap
[n_l, n_{l+1})$, then $n_l < r_{i+1}$ whence equation \eqref{eq:test2} implies that
\[
 \gamma^\prime_j=\alpha_{n_l} \ge \alpha_{r_i}=\gamma_j.
\]
Thus $\gamma^\prime_j \ge \gamma_j$ for all $j \in \mathbb N$.

$(2) \Longrightarrow (1)$. It is straightforward.

$(2) \Longrightarrow (3)$. Since $\alpha_n>0$ for all $n \in \N$,
let us define $\bar n_i=r_i$ and
let $\{\gamma_i\}$ as before. The sequence $\{\alpha_{\bar n_i}\}$ is clearly nonincreasing.
Using equation~\eqref{eq:test3}, we just need to prove that $\gamma_n=\bar\alpha_n$ for all $n$.
Indeed, if $n \in [\bar n_i, \bar n_{i+1})$ then
\[
 \gamma_n=\alpha_{r_i}=\alpha_{\bar n_i} \le \alpha_j
\]
for all $j < r_{i+1}=\bar n_{i+1}$.
Hence, $\gamma_n=\alpha_{\bar n_i}=\min\{\alpha_j\colon j \le n\}=\bar \alpha_n$.

$(3) \Longrightarrow (1)$. It is straightforward.

\end{proof}

\end{proof}

\section*{Acknowledgments}
The authors are grateful to NUMEC for the logistic support and to Capes-PROEX
and Fapesp (09/52379-8) for the
financial support during the visits of Fabio Zucca and Daniela Bertacchi to the Universidade de S\~ao Paulo.

\end{document}